\documentclass{article}
\usepackage[utf8]{inputenc}

\usepackage{hyperref}
\usepackage[style=alphabetic, maxnames=10, maxalphanames=4, doi=false, isbn=false]{biblatex}
\usepackage{amsmath,amsfonts}
\usepackage{amsthm}
\usepackage{amssymb}
\usepackage{amsopn}
\usepackage{xcolor}
\usepackage{tikz}
\usetikzlibrary{arrows.meta}

\usepackage{cleveref}

\theoremstyle{plain}
\newtheorem{theorem}{Theorem}[section]

\newtheorem{corollary}[theorem]{Corollary}
\newtheorem{lemma}[theorem]{Lemma} 

\theoremstyle{definition}

\DeclareMathOperator{\Div}{Div}
\DeclareMathOperator{\sgn}{sgn}

\title{On the linearised force balance condition in the analysis of atomistic dislocation models}
\author{Julian Braun and Thomas Hudson}
\date{\today}

\def\sfb{\mathsf{b}}
\def\sfe{\mathsf{e}}
\def\C{\mathbb{C}}
\def\N{\mathbb{N}}
\def\R{\mathbb{R}}
\def\d{\mathrm{d}}

\begin{document}

\maketitle


\begin{abstract}
A technical condition arising in the analysis of atomistic models for dislocations is studied. Key estimates for the far-field strain behaviour proved in \cite{ehrlacher_analysis_2016} rely on a summation-by-parts argument for linearised forces, under conditions of sufficient decay and of vanishing net force in an infinite system. In particular, the latter condition is required for an application of this theory to the standard far-field dislocation predictor, but the vanishing of the associated force sum has not been fully justified. Here, the missing verification of this condition is provided and its role in determining the far-field decay of the corrector is clarified. It is moreover shown that for more general physically compatible predictors, the linearised force sum need not vanish, leading to slower decay and potentially divergent finite-domain approximations.
\end{abstract}

\section{Introduction}
\label{sec:intro}
In the low-temperature regime, many important aspects of the behaviour of defects can be uncovered through the identification of critical points in an appropriate model potential energy surface. The study of crystalline defects through molecular statics simulations is therefore an important tool available to computational materials scientists. Practitioners use a range of optimisation algorithms on interatomic potential energies to identify these critical points and thereby provide microscopic input to predictions of macroscopic material behaviour. However, given limits on computational resources these simulations are limited to finite domains often containing between $10^3$ and $10^7$ atoms, with the precise number depending on the materials system, model complexity, and the hardware available. As such, the choice of boundary conditions for a simulation introduces a further source of numerical error in addition to error induced by the parametrisation of the interatomic potential. When studying defects that induce slowly decaying strain fields, such as dislocations and cracks, the numerical errors arising from the choice of boundary conditions can be particularly significant. This challenge has led to the development of a range of methods, such as flexible boundary conditions \cite{sinclair_flexible_1978}, lattice Green’s function methods \cite{trinkle_lattice_2008}, and atomistic--to--continuum coupling approaches such as the Quasicontinuum method \cite{luskin_atomistic--continuum_2013}.

While a range of boundary conditions have long been used in computational practice, \cite{ehrlacher_analysis_2016} provides one of the first comprehensive mathematical treatments of the effect of boundary conditions on the accuracy of finite molecular statics simulations. In particular, the main theorems therein provide error bounds on the atomistic strains and energies identified through relaxation of appropriate finite-size system of atoms centred at a point defect or dislocation of interest. Furthermore, this work developed the first `regularity theory' for the far-field displacements around defects, ideas which have been developed further and exploited numerically in a range of subsequent works \cite{buze_analysis_2019,braun_effect_2019,buze_analysis_2020,braun_asymptotic_2022,braun_higher-order_2025,braun_incompleteness_2025,duque_lopez_bridging_2026, wei2025higherorderboundaryconditionsatomistic}.

A key idea underlying the analytical results of \cite{ehrlacher_analysis_2016} is the use of a linearisation of atomic interactions in the far-field. This is possible since the strains induced by a defect decay away from the defect core. This opens up the use of a range of techniques to estimate the residuals which arise from the truncation of the problem to a finite domain, and ultimately enables the characterisation of asymptotic errors as the size of the atomic domain increases.
A key technical result that underpins much of the analysis is \cite[Corollary~1]{ehrlacher_analysis_2016}, which states that a force field $f$ can be `summed by parts' if it is known both that $f$ decays sufficiently quickly and that the total sum of these forces vanishes. More specifically, it shows that it is possible to find an atomistic stress field $g$ whose discrete divergence is $f$, and that $g$ moreover decays at a faster rate than $f$. This result is applied at several key points throughout the analysis to the boundary predictor fields used; notably, it is employed in proving \cite[Lemma~16]{ehrlacher_analysis_2016} concerning the far field decay of the corrector fields around dislocations. However, in the proof of this lemma, the crucial zero force sum condition
\[\sum_{\ell \in \Lambda} f(\ell) =0\] needed to apply Corollary~1 is not carefully verified.

In light of this gap, the present work provides the missing argument to complete this proof, verifying that the standard dislocation predictor used as a boundary condition does indeed satisfy the zero force sum condition. For a pure screw dislocation a more straightforward argument can be made, and this is already contained in the results of \cite{braun_asymptotic_2022}. However, the approach taken there cannot be applied to edge or mixed dislocations, and so the results presented here identify the additional requirements on the dislocation predictor necessary to achieve force balance in the presence of a residual stress field.

Given the more recent developments in our understanding of the theory, this also provides an opportunity to highlight the physical interpretation of this force balance condition. Indeed, we show that for more general but physically-compatible choices of far-field predictor, the sum need not vanish. In such cases, the decay of displacements may be significantly worse than that predicted by the analysis in \cite{ehrlacher_analysis_2016}, leading to much larger numerical errors in finite simulation boxes and possible divergence of the approximation. Nevertheless, we will show that such issues can be circumvented through an appropriate modification of the predictor field through the addition of terms which are multiples of the elastic Green's function.

\paragraph{Outline.} The remainder of the paper is structured as follows: In Section~\ref{sec:prelims}, we recall the notation used in \cite{ehrlacher_analysis_2016}, and provide details of the geometric set-up of the domain and the energy used in the analysis. In Section~\ref{sec:main_result}, we provide a precise statement of our main result, and discuss its consequences; Section~\ref{sec:proof} is then devoted to the proof of the result.

\section{Preliminaries}
\label{sec:prelims}

\subsection{Geometric set-up}
\label{sec:geometry}
The geometric setting for the problem we study is illustrated in Figure~\ref{fig:geometry}. In particular, we consider a two-dimensional reference Bravais lattice $\Lambda\subset\R^2$ which is the orthogonal projection of a three-dimensional Bravais lattice. We denote points in the 2D lattice $\ell\in\Lambda$. Throughout, we will use subscripts to indicate the Cartesian coordinates of points, i.e. $\ell = (\ell_1,\ell_2)\in\R^2$.

\begin{figure}[b!]
\centering
\begin{tikzpicture}[scale=0.8]
\def\r{1.4}
\def\L{4}

\begin{scope}
\clip (0,-\L) rectangle (\L,\L);
\fill[gray!30, even odd rule]
  (0,-\L) rectangle (\L,\L)
  (0,0) circle (\r);
\textbf{\node at ({0.75*\L},{0.75*\L}) {$\Omega_\Gamma$};}
\end{scope}

\draw[->] (-\L,0) -- ({1.1*\L},0);
\draw[->] (0,-\L) -- (0,{1.1*\L});

\draw[ultra thick,red] (0,0) -- (\L,0);
\node[above left,red] at (0.55*\L,0) {$\Gamma$};

\fill (0,0) circle (3pt);
\node[above left] at (0,0) {$\hat{x}$};

\foreach \i in {-4,...,3}
{
  \foreach \j in {-4,...,3}
  {
    \fill[blue!50] (\i+0.5,\j+0.5) circle (3pt);
  }
}
\node[blue] at ({-0.75*\L},{0.75*\L}) {$\Lambda$};

\end{tikzpicture}
\caption{The geometric setup. The defect centre is shown at $\hat{x}$; in blue, the lattice $\Lambda$; in red, the branch cut $\Gamma$, and in gray, the region $\Omega_\Gamma$, which is the right-hand half plane with a ball of radius $r$ about $\hat{x}$ removed. As shown, this radius may be larger the lattice spacing.}
\label{fig:geometry}
\end{figure}
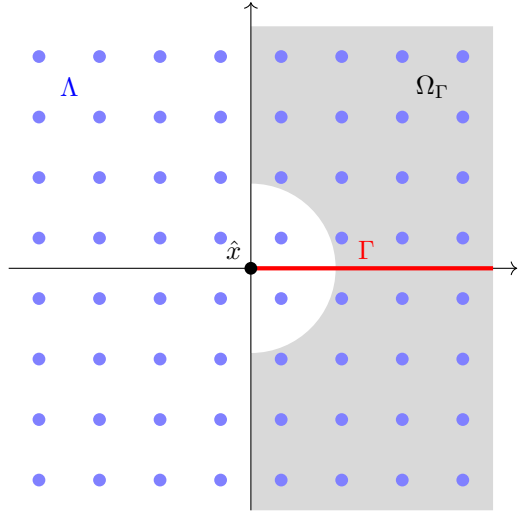
Following the setup in \cite{ehrlacher_analysis_2016} we consider a dislocation centred at a point $\hat{x} = (\hat{x}_1,\hat{x}_2)\in\R^2$; this point represents the position of the dislocation line as it cuts the the two-dimensional plane in the continuum fields we define. Let $\Gamma\subset\R^2$ be a branch cut which extends from $\hat{x}$ to the right, i.e.
\[
\Gamma = \{(x_1,x_2)\in\R^2:x_2=\hat{x}_2,x_1\geq\hat{x}_1\},
\]
and set $\Omega_\Gamma$ to be the right-hand half plane with a disc of radius $r$ about $\hat{x}$ removed, so
\[
\Omega_\Gamma = \{(x_1,x_2)\in\R^2:x_1\geq \hat{x}_1\}\setminus B_{r}(\hat{x}).
\]

Throughout, we use the notation $\sfb = (b_1,b_2,b_3)$ to indicate the possibly three-dimensional Burgers vector of the dislocation. Following the convention in \cite{ehrlacher_analysis_2016} we write $\sfb_{12}$ to indicate the projection of the Burgers vector onto its first two components, i.e. $\sfb_{12} = P_{\mathrm{2D}}\sfb = (b_1,b_2)$. Since is is a topological requirement that $\sfb$ must be a lattice vector, $\sfb_{12}$ must lie in a plane within the reference lattice we consider. Without of loss generality therefore, we are free to rotate the coordinate system such that the $\sfe_1$ direction lies parallel to a family of lattice planes, and such that the component $b_2$ vanishes; this again follows the convention used in \cite{ehrlacher_analysis_2016}. As is standard in the materials science literature, we will occasionally refer to the `edge' and `screw' components of the Burgers vector: in our geometric setup therefore, the screw component is $b_3$, and the edge component is $b_1$. We note that since the case of a pure screw dislocation where $b_1=0$ is already covered by the analysis of \cite{braun_asymptotic_2022}, we focus on the case where $b_1\neq 0$ throughout.

\subsection{Continuum predictor}
To impose the correct behaviour in the far-field away from the dislocation, we introduce a predictor displacement field derived from linear elasticity theory, denoted $u_0:\R^2\to\R^2$. To define this linear elastic predictor, the analysis in \cite{ehrlacher_analysis_2016} begins by seeking the solution to the following elastic force balance problem:
\begin{equation}
\begin{aligned}
\nabla\cdot\left(\C\nabla u^{\text{lin}}\right) &= 0&\quad x&\in\R^2\setminus\Gamma,\\
\big[u^\text{lin}(x)\big]^+_- & = -\sfb &x&\in\Gamma\setminus\{\hat{x}\},\\
\big[(\C\nabla u^{\text{lin}})\mathsf{e}_2\big]^+_- &= 0&x&\in\Gamma\setminus\{\hat{x}\},
\end{aligned}
\label{eq:ulin_problem}
\end{equation}
where the notation $[\,\cdot\,]^+_-$ denotes the jump of a given field across $\Gamma$: in particular, for $x=(x_1,0)\in\Gamma$,
$$
[v(x)]^+_- := \lim_{x_2\to 0+} v(x_1,x_2)-\lim_{x_2\to 0-} v(x_1,x_2).
$$
The natural choice of continuum linear elasticity (CLE) tensor $\C$ arising in our atomistic setting is discussed in more detail below in section \Cref{sec:energy}. We point out at this stage that $u_0\neq u^{\mathrm{lin}}$, but the predictor is constructed from $u^{\mathrm{lin}}$, as we will see below.

The PDE system in the bulk is simply the standard force balance for linear elasticity. The first boundary condition on the branch cut $\Gamma$ enforces that $u^{\text{lin}}$ has a jump corresponding to the Burgers vector of the dislocation, and the second boundary condition enforces that there is no jump in the normal stress across $\Gamma$.

As this PDE problem is posed in a domain with a cusp at $\hat{x}$ due to the removal of the ray $\Gamma$, the conditions above are in fact insufficient to determine a unique solution in $\mathrm{C}^\infty(\R^2\setminus\Gamma)$. Indeed, given a solution to this problem, we are free to add any multiple of the elastic Green's function (or its derivatives) centred at $\hat{x}$, and all of the conditions above will continue to be satisfied. Adding multiples of the Green's function and its derivatives centred at $\hat{x}$ in this way is equivalent to adding a force line or force multipoles along the dislocation line. A further possibility is to add any linear displacement of the form $Ax$ under the condition that displacement preserves the jump conditions across $\Gamma$. In an isotropic case, a class of simple example strains which preserve the solution is to add strains which uniformly expand or compress the lattice, i.e. where $A=\alpha I$ for some $\alpha\in\R$.

As in the case of modelling fracture, there is however a standard canonical solution to this problem, where we require (1) that there are no force lines or multipoles at the dislocation line, and (2) that the resulting strains decay at infinity. The former condition can be expressed as requirement that
$$
\lim_{r\to0}\int_{\partial B_r(\hat{x})} x^{\otimes n}\otimes\sigma^{\text{lin}}:\nu\,\mathrm{d}x = 0^{\otimes (n+1)},\quad\text{where}\quad \sigma^{\text{lin}} = \C\nabla u^{\text{lin}}.
$$
Here, $x^{\otimes n}$ means the $n$-fold tensor product of the position vector $x$ with itself and $\nu$ denotes the outward-pointing unit normal field on $\partial B_r(\hat{x})$. The second condition can be enforced by requiring that total strain gradient $\lvert\nabla u^{\text{lin}}(x)\rvert\to 0$ as $|x|\to\infty$ within $\R^2\setminus\Gamma$.

Under these additional conditions, the resulting solution $u^{\text{lin}}$ has a logarithmic singularity at $\hat{x}$. This canonical solution (which is known explicitly in the isotropic case) is the one which is usually reported in the standard literature on dislocations, see for example \cite[Chapter~2]{hirth_theory_1992}. Moreover, we also have the estimates
\begin{equation}\label{eq:ulin_derivative_bound}
    \lvert\nabla^j u^{\text{lin}}(x)\rvert \leq C_j |x|^{-j}\quad\text{for all }x\in\R^2\setminus\Gamma,
\end{equation}
for some fixed constant $C_j$, and for any $j\in\N$; these decay properties can be proven by using a Fourier characterisation of $u^{\mathrm{lin}}$, for example based upon the approach outlined in \cite{bacon_anisotropic_80}. We will discuss the possibility of other reference solutions where we relax the canonical requirements when interpreting our main result.

For our analysis, we will require smoothness of the continuum predictor under the operation of a particular lattice shift operator we define below. As defined, $u^{\text{lin}}$ does not satisfy this property, and so we will modify the canonical solution to address this. To do so, we following \cite{ehrlacher_analysis_2016} in defining the mapping $\xi:\R^2\setminus\Gamma\to\R^2$ by
$$
\xi(x) := x-\frac{\sfb_{12}}{2\pi}\arg(x-\hat{x})\eta\big(\tfrac{|x-\hat{x}|}{r}\big),
$$
where $\eta:\R\to[0,1]$ is $\mathrm{C}^\infty$, and satisfies $\eta(s) = 0$ for $s<0$, $\eta(s) = 1$ for $s>1$ and $\eta'(s)>0$ for $0<s<1$.
With this definition, $\xi$ is a smooth diffeomorphism, except where it induces a (smoothly varying) jump across the ray $\{x\in\R^2 : x_2=\hat{x}_2,x_1\leq \hat{x}_2\}$; the same is therefore true for its inverse, $\xi^{-1}$. By construction, we have that there exist global constants $C_j>0$ such that
\begin{equation}\label{eq:xi_derivative_bound}
\begin{aligned}
\lvert\xi^{-1}(x)-x\rvert &\leq C_0\quad\text{for all }x\in\R^2\setminus\Gamma,\\
\lvert\nabla^j\xi^{-1}(x)\rvert &\leq C_j\quad\text{for all }x\in\R^2\setminus\Gamma,
\end{aligned}
\end{equation}
where the latter estimate holds for any $j\in\N$; in practice we will use only $j\in\{1,2,3\}$.

Using the definition of $\xi$, we are now in a position to define the predictor we will use as
\begin{equation}
    u_0:=u^{\mathrm{lin}}\circ \xi^{-1};
    \label{eq:predictor}
\end{equation}
this choice follows exactly the approach of \cite{ehrlacher_analysis_2016}.

\subsection{Lattice operators}\label{sec:operators}
We next recall the definitions of the various lattice shift and finite difference operators as used in \cite{ehrlacher_analysis_2016}. We define unscaled finite differences for lattice functions via
\[
D_\rho u(\ell) := u(\ell+\rho)-u(\ell),
\]
where $\rho$ is a lattice direction. We also define lattice shift operators acting on lattice functions as follows:
\begin{equation*}
\begin{aligned}
  S_0 u(\ell) &:= \begin{cases}
    u(\ell) &\ell_2>\hat{x}_2\\
    u(\ell-\sfb_{12})-\sfb & \ell_2<\hat{x}_2,
  \end{cases}\\
  S u(\ell) &:= \begin{cases}
    u(\ell) &\ell_2>\hat{x}_2\\
    u(\ell-\sfb_{12}) \hphantom{-\sfb\;} & \ell_2<\hat{x}_2,
  \end{cases}\\
  R u(\ell)&:=\begin{cases}
    u(\ell) &\ell_2>\hat{x}_2\\
    u(\ell+\sfb_{12}) \hphantom{-\sfb\;} & \ell_2<\hat{x}_2.
  \end{cases}
  \end{aligned}
\end{equation*}
Using these operators, we also define the shifted finite difference operator as follows:
\[\tilde{D}_\rho u(\ell) :=
\begin{cases}
    RD_\rho S u(\ell) & \ell \in \Omega_\Gamma \\
    D_\rho u(\ell) & \text{otherwise.}
\end{cases}
\]
Given the lattice operators $D_\rho$ and $\tilde{D}_\rho$, we can define operators $\Div$ and $\widetilde{\Div}$ by duality, i.e. given $g_\rho(\ell)$ defined for all $\ell\in\Lambda$ and $\rho\in\mathcal{R}$, we write $f = \Div g$ and $\tilde{f} = \widetilde{\Div}\,g$ if
\[
\sum_{\ell\in\Lambda}f(\ell)u(\ell) = \sum_{\ell\in\Lambda}\sum_{\rho}g_\rho(\ell)D_\rho u\quad\text{and}\quad \sum_{\ell\in\Lambda}\tilde{f}(\ell)u(\ell) = \sum_{\ell\in\Lambda}\sum_{\rho}g_\rho(\ell)\tilde{D}_\rho u.
\]
It can be shown immediately that these operators are well-defined on all compactly-supported $g$, and the domain of definition can be extended by continuity to any appropriate spaces in which such functions are dense.

\subsection{Special strain}\label{sec:specialstrain}
Applying the shifted difference operator $\tilde{D}_\rho$ to $u_0$, the continuum predictor for the displacement field as defined in \eqref{eq:predictor}, we define the \emph{special strain field} to be
\[
e_\rho(\ell) := 
\begin{cases}
    RD_\rho S_0 u_0(\ell) & \ell \in \Omega_\Gamma \\
    D_\rho u_0(\ell) & \text{otherwise.}
\end{cases}
\]
We note that this definition uses the operator $S_0$ in $\Omega_\Gamma$, while the definition $\tilde{D}$ uses $S$ in the same region, so $e_\rho$ is not simply $\tilde{D}_\rho u_0$.

As we will see through the result of \Cref{lem:finitediffops} below, $e_\rho (\ell)$ is in fact identical to the finite difference $D_\rho u_0(\ell)$ unless the line segment $(\ell,\ell+\rho)$ crosses the branch cut $\Gamma\cap\Omega_\Gamma$. In the latter case a shift is performed, which leads to the estimate $\lvert e_\rho(\ell)\vert \leq C \lvert \ell \rvert^{-1}$, even around $\Gamma$. As such, $e_\rho$ encodes a reference elastic strain, rather than the total strain $D_\rho u(\ell)$, and the operators in the definition can be viewed heuristically as allowing us to compare the displacements relative to a local reference: A global reference cannot exist due to the incompatibility induced by the presence of the dislocation. We note that due to the constant jump in $u_0$ across $\Gamma$ required to create the edge component of the dislocation, the decay estimate for $e_\rho(\ell)$ cannot hold for the finite difference (i.e. the total strain) itself.

\subsection{Elastic Energy} \label{sec:energy}

For a displacement $u \colon \Lambda \to \R^2$ we consider the energy
\[\mathcal{E}(u) = \sum_{\ell \in \Lambda} V_{\ell}(Du(\ell)) \]
with a site potential $V_\ell$ that is independent of $\ell$ outside of a compact core region, so $V=V_\ell$. The site potentials are assumed to be sufficiently smooth.

To enable us to make proper sense of dislocations, we furthermore assume the shift invariance of the site potentials under a slip in the Burgers vector direction, i.e.
\begin{align*}
    V(DS_0 v (\ell)) &= V(D v (\ell)) &\text{ for } \ell_2 > \hat{x}_2,\\
    V(DS_0 v (\ell+\sfb_{12})) &= V(D v (\ell)) &\text{ for } \ell_2 < \hat{x}_2.
\end{align*}
Physically, this is a natural consequence of the permutation invariance of atoms within the lattice.

With the site potentials defined, the resulting continuum Cauchy-Born energy density is then directly given as
\[ W(A) = c_{\Lambda} V((A \rho)_\rho)\]
where $c_\Lambda$ is the inverse volume per atom in the reference lattice. In this framework, the CLE tensor at the reference lattice is given by the second derivative of the energy density at the reference lattice, $\C := D^2 W(0)$.

\section{Main results}
\label{sec:main_result}
We now state our main results. In order to do so, we introduce the following tensor, $\sigma^0$, which in components is
\[
\sigma^0_{ij}:= \nabla W(0)_{ij} = c_\Lambda\sum_{\rho} \nabla_\rho V(0)_{i}\rho_j.
\]
Here, $\nabla_\rho V(0)$ is the partial derivative of the site energy with respect to the displacement difference $D_\rho u(\ell)$ at $D_\rho u\equiv 0$. Note that the units of $\sigma^0$ are those of stress, and $\sigma^0$ therefore encodes any residual stress that might be present when using the energy with the reference lattice selected. It is typical in atomistic modelling to perform a relaxation of the lattice such that this stress vanishes unless considering some macroscopic strain is applied to the system; our analysis however allows us to consider the possibility where this is not the case, which is likely to be of interest for some applications.

To discuss the result in a clear fashion, let us fix the line direction $\tau=\sfe_3$ of the full three dimensional setting and $n = P_{\text{2D}}\frac{\tau\times \sfb}{|\tau\times \sfb|}$, where we recall that $P_{\mathrm{2D}}$ is the projection of a 3D vector onto its first two components. The unit vector $n$ has the interpretation of being the normal to the slip plane on which a dislocation with a non-zero edge component is free to glide. Notice that thanks to our geometric set up, $n=\pm\sfe_2$, depending on the sign of the edge component of the Burgers vector $b_1$. We can now state our main results.

\begin{theorem} \label{thm:netresforce}
Defining the forces at the reference displacement
\[
f(\ell) = -\Div\big(\delta\mathcal{E}(u_0)\big)(\ell),
\]
we have that the sum of these forces can be expressed as
\[
    \sum_{\ell\in\Lambda}f(\ell) = -\sigma^0 n |\sfb_{12}| - \int_{\partial B_1(0)}\C[\nabla u^{\mathrm{lin}}]\cdot \nu \,\d\sigma.
    \]
In particular, the sum of the forces vanishes if the two terms on the right hand side balance; that is, if $u^{\mathrm{lin}}$ is chosen such that
\[-\sigma^0 n |\sfb_{12}| = \int_{\partial B_1(0)}\C[\nabla u^{\mathrm{lin}}]\cdot \nu \,\d\sigma.\]
\end{theorem}

\noindent
A particularly important special case for practical purposes is the following corollary.

\begin{corollary} \label{cor:zeronet}
If there is no residual stress at the reference lattice, so that $\sigma^0=0$, and there is no force line at the dislocation core, meaning $\int_{\partial B_1(0)}\C[\nabla u^{\mathrm{lin}}]\cdot \nu \,\d \sigma=0$, then
\[
    \sum_{\ell\in\Lambda}f(\ell) = 0.
\]
\end{corollary}

To conclude the presentation of our results, we now make some comments about their interpretation.

We note first that the two terms on the right-hand side of the expression for the sum of the linearised forces in \Cref{thm:netresforce} each has a different origin.
The first of these terms, involving $\sigma^0$, can be interpreted as the `residual' force line at the dislocation core which arises as a result of being present in a lattice under macroscopic stress. The second integral term, encoding the total linearised elastic stress which results macroscopically from applying the predictor $u^{\mathrm{lin}}$, can be used to cancel this residual stress to make the force sum zero. An inspection of the proof shows that the choice of the contour of integration in this second term is arbitrary; the only requirement is that it is a simple closed curve encircling the dislocation line: in the statement, we fix it to be a ball of radius 1 for simplicity.

In the general case where the residual stress does not vanish, we can achieve cancellation of the two contributions by choosing $u^{\mathrm{lin}}$ to contain a `force line' contribution alongside the standard CLE dislocation predictor. For example, if $u^{\mathrm{fl}}$ is a distributional solution to
\[
\nabla\cdot \left(\C\nabla u^{\mathrm{fl}}\right) = -\delta_0 \sigma^0n |\sfb_{12}|,
\]
where $\delta_0$ is a Dirac delta centred at the dislocation core,
then it straightforward to verify that indeed
\[-\sigma^0 n|\sfb_{12}| = \int_{\partial B_1(0)}\C[\nabla u^{\mathrm{fl}}]\cdot \nu \,\d\sigma,
\]
and so upon adding this to $u^{\mathrm{lin}}$ as defined in \cref{eq:ulin_problem}, we can achieve balance balance. Note that $u^{\mathrm{fl}}$ is simply an appropriate multiple of the (two-dimensional) CLE Green's function, and can therefore be constructed with the same rate of strain decay as the dislocation predictor terms. Since both of these contributions have the same rate of decay, the inclusion of such terms is required at leading order in order to make sense of much of the analysis in \cite{ehrlacher_analysis_2016}. As noted in the introduction, without the inclusion of such terms, using the predictor $u^{\mathrm{lin}}$ alone as a boundary condition will lead to divergence of any finite-dimensional numerical approximation.

At first inspection, it might seem natural to attribute the force line contribution identified by \Cref{thm:netresforce} to a Peach-Koehler force acting on the dislocation in the presence of a residual stress field, as it has the same physical dimensions. However, considering a pure screw dislocation shows that this cannot be the case: For such dislocations, the projection of the Burgers vector onto the 2D plane $\sfb_{12}$ is zero, and so the first term on the right vanishes even if $\sigma^0$ is non-zero. In contrast, screw dislocations can indeed experience Peach-Koehler forces when subject to shear stresses, and so we cannot draw a direct connection between the Peach-Koehler force and this extra requirement.

Finally, we note that the case in which the residual stress vanishes, \Cref{cor:zeronet} provides an important standard case. As mentioned, an initial relaxation of the atomic lattice is a standard step in atomistic simulation of defects. If this step is performed, we see that this removes the need to counteract any force line contribution automatically.

\section{Proofs and Calculations}
\label{sec:proof}

\subsection{Auxiliary results}
As a first step towards the proof of \Cref{thm:netresforce}, the following lemma provides two important auxiliary results.

\begin{lemma} \label{lem:finitediffops}
    With the definitions of the operators $D$ and $\tilde{D}$ and the special strain $e$ given in \Cref{sec:operators,sec:specialstrain}, we find
    \begin{align*}
    e_\rho(\ell) &= \begin{cases}
      D_{\rho+\sfb_{12}}u_0(\ell) + \mathsf{b}\qquad\, & \ell\in\Omega_\Gamma,\ell_2<\hat{x}_2<(\ell+\rho)_2,\\
      D_{\rho-\sfb_{12}}u_0(\ell) - \mathsf{b} & \ell\in\Omega_\Gamma,(\ell+\rho)_2<\hat{x}_2<\ell_2,\\
      D_\rho u_0(\ell) & \text{otherwise.}
    \end{cases}\quad\text{and}\\
    \tilde{D}_\rho u(\ell)-D_\rho u(\ell) &= \begin{cases}
      D_{\rho+\sfb_{12}}u(\ell)-D_{\rho}u(\ell) & \ell\in\Omega_\Gamma, \ell_2<\hat{x}_2<(\ell+\rho)_2\\
      D_{\rho-\sfb_{12}}u(\ell)-D_\rho u(\ell)& \ell\in\Omega_\Gamma, (\ell+\rho)_2<\hat{x}_2<\ell_2\\
       0 & \text{otherwise.}
        \end{cases}
\end{align*}
\end{lemma}
\begin{proof}
    We consider a general displacement $u$ and directly compute $D_\rho S_0 u(\ell)$ in the various cases possible based on the positions of $\ell$ and $\ell+\rho$. This results in the following expressions:
\begin{equation*}
    D_\rho S_0 u(\ell) = \begin{cases}
      D_\rho u_0(\ell) & \ell_2,(\ell+\rho)_2>\hat{x}_2\\
      D_\rho u(\ell-\sfb_{12}) & \ell_2,(\ell+\rho)_2<\hat{x}_2\\
      D_{\rho+\sfb_{12}}u(\ell-\sfb_{12})+\sfb & \ell_2<\hat{x}_2<(\ell+\rho)_2\\
      D_{\rho-\sfb_{12}}u(\ell)-\sfb&
      (\ell+\rho)_2<\hat{x}_2<\ell_2.
    \end{cases}
\end{equation*}
Applying the operator $R$ to this result, and noting our standing assumption that $\sfb_{12}$ is to be parallel to $\sfe_1$, so no additional cases arise, we find:
\begin{equation*}
    R D_\rho S_0 u(\ell) = \begin{cases}
      D_\rho u(\ell) & \ell_2,(\ell+\rho)_2>\hat{x}_2\\
      D_\rho u(\ell) & \ell_2,(\ell+\rho)_2<\hat{x}_2\\
      D_{\rho+\sfb_{12}}u(\ell) +\sfb & \ell_2<\hat{x}_2<(\ell+\rho)_2\\
      D_{\rho-\sfb_{12}}u(\ell)-\sfb&
      (\ell+\rho)_2<\hat{x}_2<\ell_2.
    \end{cases}
\end{equation*}
Setting $u=u_0$, this calculation directly proves the claim about $e_\rho$.

Performing the same calculation using the operator $S$ in place of $S_0$, the shifts involving $\sfb$ cancel, and we have
\begin{equation*}
    \tilde{D}_\rho u (\ell) = R D_\rho S u(\ell) = \begin{cases}
      D_\rho u(\ell) & \ell_2,(\ell+\rho)_2>\hat{x}_2\\
      D_\rho u(\ell) & \ell_2,(\ell+\rho)_2<\hat{x}_2\\
      D_{\rho+\sfb_{12}}u(\ell) & \ell_2<\hat{x}_2<(\ell+\rho)_2\\
      D_{\rho-\sfb_{12}}u(\ell)&
      (\ell+\rho)_2<\hat{x}_2<\ell_2.
    \end{cases}
\end{equation*}
Subtracting $D_\rho u$ from these expressions now yields the second result.
\end{proof}

\subsection{Proof of Theorem \ref{thm:netresforce}}
Recall from \Cref{eq:predictor} that $u_0=u^{\mathrm{lin}}\circ \xi^{-1}$. Applying the uniform estimates \cref{eq:ulin_derivative_bound} and \cref{eq:xi_derivative_bound}, the chain rule allows us to deduce that
\begin{equation*}
    \lvert \nabla^j u_0 \lvert \lesssim \lvert x\rvert^{-j}\quad\text{for all }x\in\R^2\setminus\Gamma,     
\end{equation*}
and for $j=1,2,3$. Furthermore, for any finite set of directions $\rho$, we can use the results of \Cref{lem:finitediffops} to assert that there exists a constant $C$ such that
\begin{equation}\label{eq:strain_bound}
    \lvert e_\rho(\ell)\vert \leq C \lvert \ell \rvert^{-1},
\end{equation}
for any $\ell\in\Lambda$ and any $\rho$, even when the line segment $(\ell,\ell+\rho)$ crosses $\Gamma$.

Take a parametrised family of smooth cut-off functions $\eta_R \colon \R^2 \to [0,1]$ with $\eta_R(x)=1$ for $\lvert x \rvert \leq R$, $\eta_R(x)=0$ for $\lvert x \rvert > 2R$ and $\lvert \nabla \eta_R(x) \rvert\lesssim R^{-1}$. For technical reasons which will become apparent later, we additionally make the further requirement that $\frac{\partial\eta_R}{\partial x_2} = 0$ in a neighbourhood of the branch cut $\Gamma$. Taking components, we now compute:
\begin{align*}
\sum_{\ell \in \Lambda} f_i(\ell) &= \sum_{\ell \in \Lambda} \big(-\Div\delta \mathcal{E}(u_0)\big)_i \\
&= \sum_{\ell \in \Lambda} \Big(-\widetilde{\Div} \nabla V_\ell\big(e(\ell)\big)\Big)_i\\
&= \lim_{R \to \infty} \sum_{\ell \in \Lambda} \eta_R \Big(-\widetilde{\Div} \nabla V_\ell\big(e(\ell)\big)\Big)_i.
\end{align*}
Summing by parts, which is possible since $\eta_R$ is compactly supported, we obtain
\begin{equation*}
\sum_{\ell \in \Lambda} \eta_R \Big(-\widetilde{\Div} \nabla V_\ell\big(e(\ell)\big)\Big)_i = \sum_{\ell \in \Lambda} \nabla V(e(\ell))[\sfe_i \otimes \tilde{D} \eta_R].
\end{equation*}
Adding and subtracting terms, we then write
\begin{equation}\label{eq:splitting}
\begin{aligned}
&\sum_{\ell \in \Lambda} \nabla V(e(\ell))[\sfe_i \otimes \tilde{D} \eta_R]\\
&\qquad= \sum_{\ell \in \Lambda} \Big(\nabla V(e(\ell)) - \nabla V(0) - \nabla^2 V(0)[e(\ell)]\Big)[\sfe_i \otimes \tilde{D} \eta_R]\\
&\qquad\qquad +\sum_{\ell \in \Lambda} \nabla^2 V(0)[e(\ell), \sfe_i \otimes \tilde{D} \eta_R]+ \sum_{\ell \in \Lambda} \nabla V(0)[\sfe_i \otimes \tilde{D} \eta_R].
\end{aligned}
\end{equation}
We now analyse each of the terms on the right-hand side of \cref{eq:splitting} in turn.

\paragraph{First term.} The first of the terms on the right-hand side is a linearisation error, and may be estimated as follows:
\begin{align*}
&\left\lvert \sum_{\ell \in \Lambda} \Big(\nabla V(e(\ell)) - \nabla V(0) - \nabla^2 V(0)[e(\ell)]\Big)[\sfe_i \otimes \tilde{D} \eta_R] \right\rvert\\
&\qquad\lesssim \!\!\sum_{\lvert \ell \rvert \in(\frac12R, \frac52R)} \lvert e(\ell) \rvert^2 \lvert \tilde{D} \eta_R \rvert\lesssim \!\!\sum_{\lvert \ell \rvert \in(\frac12R, \frac52R)} \lvert \ell \rvert^{-2}R^{-1}\lesssim R^{-1},
\end{align*}
where we use the assumption that $|\nabla\eta_R|\lesssim R^{-1}$, the bound \cref{eq:strain_bound}, and the result of \Cref{lem:finitediffops} to bound $\tilde{D}_\rho\eta$. We see as a result that this term vanishes in the limit $R\to\infty$.

\paragraph{Second term.} For the second term on the right-hand side of \cref{eq:splitting}, we again use the facts that $\tilde{D} \eta_R(\ell)$ and $\nabla \eta_R(\ell)$ are only non-zero for $\ell$ such that $\tfrac12R \leq \lvert \ell \rvert <\tfrac52R$, allowing us to see that
\begin{equation*}
\sum_{\ell \in \Lambda} \nabla^2 V(0)[e(\ell), \sfe_i \otimes \tilde{D} \eta_R]= \int_{\R^2 \backslash \Gamma} \C[\nabla u_{0}, \sfe_i \otimes \nabla \eta_R] \,\d x + O(R^{-1}).
\end{equation*}
The error term on the right above arises due to the fact that the leading order contribution to the discretisation error in the integrand is of the form $|\ell|^{-2}R^{-1}$, and when integrated over the appropriate annulus, this gives us an error term of the stated order in $R$.
As such, this establishes that
\begin{equation*}
\lim_{R \to \infty} \sum_{\ell \in \Lambda} \nabla^2 V(0)[e(\ell), \sfe_i \otimes \tilde{D} \eta_R]=\lim_{R \to \infty} \int_{\R^2 \backslash \Gamma} \C[\nabla u_{0}, \sfe_i \otimes \nabla \eta_R] \,\d x.
\end{equation*}
Next, in light of the bounds \cref{eq:ulin_derivative_bound,eq:xi_derivative_bound}, we have that
\[ \left\lvert \nabla u_{0}(x) - \nabla u^{\rm lin}\big(\xi^{-1}(x)\big) \right\rvert \lesssim \lvert x \rvert^{-2}.
\]
This allows us to change variable in the integral on the right above, giving:
\begin{align*}
\int_{\R^2 \backslash \Gamma} \C[\nabla u_{0}, \sfe_i \otimes \nabla \eta_R] \,\d x
&= \int_{\R^2 \backslash \Gamma} \C[\nabla u^{\rm lin}(\xi^{-1}(x)), \sfe_i \otimes \nabla \eta_R] \,\d x\\
&=\int_{\R^2 \backslash \Gamma} \C[\nabla u^{\rm lin}(x), \sfe_i \otimes \nabla \eta_R(\xi(x))] \lvert \mathrm{det} D\xi(x)\rvert \,\d x.
\end{align*}
Next, we note that as a result of the properties of $\eta_R$ and the diffeomorphism $\xi$, we have $|\nabla\eta_R(\xi(x))-\nabla\eta_R(x)|\lesssim R^{-2}$ and $|\det D\xi(x)-1|\lesssim |x-\hat{x}|^{-1}$. Using the compact support of the difference $\nabla\eta_R(\xi(x))-\nabla\eta_R(x)$ and sending $R\to\infty$,
\begin{equation*}
    \lim_{R\to\infty}\int_{\R^2 \backslash \Gamma} \C[\nabla u_{0}, \sfe_i \otimes \nabla \eta_R] \,\d x
=\lim_{R \to \infty} \int_{\R^2 \backslash \Gamma} \C[\nabla u^{\rm lin}(x), \sfe_i \otimes \nabla \eta_R(x)] \,\d x.
\end{equation*}
Now, using the compactness of the support of $\eta_R$ and its derivative once more, we restrict the integral to $\R^2\setminus B_1(\hat{x})$ where we will assume that $1<R$ to permit us to make the argument below. Using this restriction, integrating by parts on this domain, and using that the linear elastic stress induced by the predictor $\C \nabla u^{\mathrm{lin}}$ is continuous across $\Gamma$, we find that
\begin{align}\label{eq:term2}
&\int_{\R^2 \backslash \Gamma} \C[\nabla u^{\rm lin}, \sfe_i \otimes \nabla \eta_R] \,\d x\notag\\
&\qquad=\int_{\R^2 \setminus B_1(0)} \C[\nabla u^{\rm lin}, \sfe_i \otimes \nabla \eta_R] \,\d x\notag\\
&\qquad=\int_{\R^2 \backslash B_1(\hat{x})} -\mathrm{div}\, \C[\nabla u^{\rm lin}]\cdot\sfe_i\, \eta_R \, \d x - \int_{\partial B_1(\hat{x})} \Big(\C[\nabla u^{\rm lin}]\sfe_i\Big) \cdot\nu \,\d\sigma\notag\\
&\qquad=-\int_{\partial B_1(\hat{x})}\Big(\C[\nabla u^{\rm lin}]\sfe_i\Big)\cdot \nu \,\d\sigma,
\end{align}
where $\nu$ is the normal pointing outwards on $\partial B_1(\hat{x})$ (so the inward-pointing normal for the region on which the divergence theorem is applied, hence the sign). Note that the choice of a ball of radius $1$ was arbitrary: an application of the divergence theorem and the properties assumed of $u^{\mathrm{lin}}$ show that integral any simple closed contour enclosing the origin could be used and the argument would remain valid.

\paragraph{Third term.} Let us now turn to the final term in \cref{eq:splitting}. In this case, we can first use the fact that $\sum_{\ell\in\Lambda}D\eta_R = 0$ to write
\begin{align*}
 \sum_{\ell \in \Lambda} \nabla V(0)[\sfe_i \otimes \tilde{D} \eta_R]
  &=\sum_{\ell \in \Lambda} \nabla V(0)[\sfe_i \otimes (\tilde{D} \eta_R - D \eta_R)].
\end{align*}
Next, we apply the result of \Cref{lem:finitediffops}, giving
\begin{align*}
\sum_{\ell \in \Lambda} \nabla V(0)[\sfe_i \otimes \tilde{D} \eta_R]
  &=\sum_{\rho}\sum_{\ell \in \Omega_{\Gamma}, [\ell, \ell+\rho] \cap \Gamma \neq \emptyset} \nabla V(0)_{i \rho} D_{\sgn(\rho_2) \sfb_{12}} \eta_R (x+\rho),
\end{align*}
with $[\ell, \ell+\rho]$ denoting the line segment.

Up to discretisation error terms which vanish in the limit where $R\to\infty$, this sum can once more be replaced by an integral. More specifically, consider
\[ A_\rho = \{ x \in \Omega_\Gamma \colon [x, x+\rho] \cap \Gamma \neq \emptyset\}; \]
away from $\partial \Omega_\Gamma$, $A_\rho$ is just a strip of width $\lvert \rho_2 \rvert$ around $\Gamma$. We obtain
\begin{align}\label{eq:term3a}
 \sum_{\ell \in \Lambda} &\nabla V(0)[\sfe_i \otimes \tilde{D} \eta_R] \notag\\
  &=\sum_{\rho}\sum_{\ell \in \Omega_{\Gamma}, [\ell, \ell+\rho] \cap \Gamma \neq \emptyset} \nabla V(0)_{i \rho} D_{\sgn(\rho_2) \sfb_{12}} \eta_R (x+\rho)\notag\\
  &=\sum_\rho c_{\Lambda} \int_{A_\rho} \nabla V(0)_{i \rho} \nabla \eta_R (x) \cdot (\sgn(\rho_2) \sfb_{12}) \,\d x  + O(R^{-1})
\end{align}
At this point, we rely on the requirement that $\eta_R$ has been chosen such that $\frac{\partial\eta_R}{\partial x_2} =0$ in $A_\rho$, which is a neighbourhood of the branch cut $\Gamma$. Using the fact the $A_\rho$ is just a strip of width $\lvert \rho_2 \rvert$ we can then reduce the integral to a one-dimensional integral and thus find
\begin{align}\label{eq:term3}
\lim_{R \to \infty} \sum_{\ell \in \Lambda} \nabla V(0)[\sfe_i \otimes \tilde{D} \eta_R]
  \quad &=\lim_{R \to \infty} \sum_\rho c_{\Lambda} \int_1^\infty \nabla V(0)_{i \rho}\partial_1 \eta_R(x_1, \hat{x}_2) \rho_2 b_1 \,\d x_1 \notag\\
  \quad &= -\sum_\rho c_{\Lambda} \nabla V(0)_{i \rho} \rho_2 b_1\notag\\
  \quad &= -\sigma^0_{i2} b_1\notag\\
  \quad &= -(\sigma^0 n)_i |\sfb_{12}|,
\end{align}
where on the last line, we use the convention $b_2=0$.

\paragraph{Conclusion.} Our analysis has characterised each of the terms on the right-hand side of \cref{eq:splitting} has now showed that the first term in \cref{eq:splitting} as $R\to\infty$. In sum, we have showed that
\begin{equation*}
    \sum_{\ell\in\Lambda} f(\ell) = -\sigma^0 n |\sfb_{12}| - \int_{\partial B_1(0)} \C[\nabla u^{\text{lin}}]\nu\,\d\sigma,
\end{equation*}
and thus the proof of the theorem is complete. We note that \Cref{cor:zeronet} is a direct consequence of assuming that $\sigma^0=0$.

\section*{Acknowledgements}
\paragraph{Thanks.} The authors thank the International Centre for the Mathematical Sciences for providing facilities and accommodation support during a visit by TH to Edinburgh which enabled some of this work. TH also acknowledges partial support from UKRI during this work through NSF-UKRI grant, reference UKRI3611.

\paragraph{Open access.}
For the purpose of open access, the authors have applied a Creative Commons Attribution (CC-BY) licence to any Author Accepted Manuscript version arising from this submission.

\paragraph{AI usage statement.} The main results of in this manuscript were proved independently of AI, and the authors take full responsibility for their accuracy. An AI model was used to produce a first draft of the illustration in \Cref{fig:geometry}.

\printbibliography

\end{document}